\documentclass[letterpaper, 10 pt, conference]{ieeeconf}  % Comment this line out if you need a4paper
\IEEEoverridecommandlockouts                              % This command is only needed if 
\usepackage{graphics} % for pdf, bitmapped graphics files
\usepackage{epsfig} % for postscript graphics files
\usepackage{amsmath} % assumes amsmath package installed
\usepackage{amssymb}  % assumes amsmath package installed
\usepackage[english]{babel}
\usepackage{algorithm}
\usepackage{algpseudocode}
\usepackage{graphicx}
\usepackage[caption=false]{subfig}
\usepackage{cite}

\usepackage{ntheorem} 
\makeatletter
\renewtheoremstyle{plain}% Adds automatic line break, if heading is too long
  {\item{\theorem@headerfont ##1\ ##2\theorem@separator}~}
  {\item{\theorem@headerfont ##1\ ##2\ (##3)\theorem@separator}~}
\makeatother

{\theoremheaderfont{\normalfont\slshape}
 \theorembodyfont{\normalfont\slshape}
 \newtheorem{remark}{Remark}
 \newtheorem{proposition}{Proposition}}
 
\usepackage{tabularx} % table allows line-break
\usepackage{fourier} 
\usepackage{array}
\usepackage{makecell}

\usepackage{arydshln}

\newcommand{\set}[1]{\left\{#1\right\}}

\newcommand{\norm}[1]{\left\Vert#1\right\Vert}

\title{\LARGE \bf
Numerical Discretization Methods \\ for the Discounted Linear Quadratic Control Problem
}
 
\author{Zhanhao Zhang, Steen Hørsholt, 
John Bagterp Jørgensen% <-this % stops a space
}

\begin{document}

\maketitle
\thispagestyle{empty}
\pagestyle{empty}

%%%%%%%%%%%%%%%%%%%%%%%%%%%%%%%%%%%%%%%%%%%%%%%%%%%%%%%%%%%%%%%%%%%%%%%%%%%%%%%%
\begin{abstract}
This study focuses on the numerical discretization methods for the continuous-time discounted linear-quadratic optimal control problem (LQ-OCP) with time delays. By assuming piecewise constant inputs, we formulate the discrete system matrices of the discounted LQ-OCPs into systems of differential equations. Subsequently, we derive the discrete-time equivalent of the discounted LQ-OCP by solving these systems. This paper presents three numerical methods for solving the proposed differential equations systems: the fixed-time-step ordinary differential equation (ODE) method, the step-doubling method, and the matrix exponential method. Our numerical experiment demonstrates that all three methods accurately solve the differential equation systems. Interestingly, the step-doubling method emerges as the fastest among them while maintaining the same level of accuracy as the fixed-time-step ODE method.
\end{abstract}

%%%%%%%%%%%%%%%%%%%%%%%%%%%%%%%%%%%%%%%%%%%%%%%%%%%%%%%%%%%%%%%%%%%%%%%%%%%%%%%%
\section{Introduction}
\label{sec:Intro}
As one of the fundamental problems within optimal control theory, the linear-quadratic optimal control problem (LQ-OCP) is attractive due to its mathematical tractability and simplicity. This class of problems has extensive applications in engineering, economics, and other fields. However, real-world scenarios often present time delays, significantly influencing system performance and robustness~\cite{Gu2003StabilityOT}. Integrating time delays into LQ-OCPs adds complexity as control inputs depend on both historical and current system states. Considering potential risks and uncertainties in the future, incorporating discounted cost functions becomes valuable for balancing immediate gains against long-term losses when devising control strategies~\cite{fu2018risk}. The continuous nature of these problems, coupled with the complexities introduced by time delays and discounted cost functions, can render their implementation infeasible in real-world scenarios. Therefore, there arises a necessity for discretization techniques to facilitate the practical implementation of discounted LQ-OCPs with time delays.

There is rich research on the solution and discretization methods of undiscounted optimal control problems~\cite{Hendricks2008LCD,Goebel2007LQRwithControlConstraint,BAGTERPJORGENSEN2012187,Dontchev2001TheEulerApproximation,Alt2013ApproximationsofLQ}. The discretization of continuous-time undiscounted LQ-OCPs without time delays has been explored extensively in the literature~\cite{Cullum1969DiscreteApproximationstoContinuoustimeOCP,Dontchev2001TheEulerApproximation,Han2010ConvergenceofDiscretetime,Alt2013ApproximationsofLQ}. Furthermore, recent studies have introduced numerical discretization methods for deterministic and stochastic LQ-OCPs with time delays~\cite{zhaz2023LQDiscretization,zhaz2023LQDiscretizationWithDelays}. 

On the other hand, the discounted cost approach is popular when considering trade-offs between the present and future costs or rewards in control strategies, e.g., reinforced learning~\cite{wang2022system} and risk-sensitive optimal control problems~\cite{jacobson1973optimal}. The discount factor is a key component in these problems. In~\cite{wang2022system,granzotto2020finite}, authors have investigated the influence of various discount factors on the stability of the discounted optimal control problems for both linear and nonlinear systems. Besides, the 
discounted LQ-OCPs are related to other control algorithms, such as model predictive control (MPC)~\cite{grune2015using} and linear-quadratic Gaussian (LQG)~\cite{Hansen1995DiscountedLinearExp,mena2022discounted}. 
However, as far as we know, the existing literature has not explored the optimal control problems incorporating with both the discounted cost function and linear time-delay systems. Therefore, in this paper, we would like to investigate the discretization of the continuous-time discounted LQ-OCPs subject to time-delay systems. The key problems that we address in this paper: 
\begin{itemize}
    \item [1.] Formulation of differential equations systems for the discretization of discounted LQ-OCPs with and without time delays 
    \item [2.] Numerical methods for solving the resulting systems of differential equations.
\end{itemize}
This paper is organized as follows. Section~\ref{sec:FormulationoftheProblem} introduces the formulation of discounted LQ-OCPs with and without time delays and describes the differential equation systems for LQ discretization. In Section~\ref{sec:NumericalMethods}, we present three numerical methods for solving proposed systems of differential equations. Section~\ref{sec:NumericalExperiments} illustrates a numerical experiment that tests and compares the three numerical methods, and the conclusions are given in Section~\ref{sec:Conclusion}.
%% %%%%%%%%%%%%%%%

% \section{Discounted Linear-Quadratic Optimal Control Problems}
\section{Formulation of the Problems}
\label{sec:FormulationoftheProblem}
\subsection{Discounted linear-quadratic optimal control problem}
Consider the following LQ-OCP with a discounted cost function 
\begin{subequations}
\label{eq:ContinuousTime-DiscountedLQOCP}
\begin{alignat}{5}
& \min_{x,u,z,\tilde z} \quad &&\phi = \int_{t_0}^{t_0+T} l_c(\tilde z(t)) dt \\
& s.t. && x(t_0) = \hat x_0, \\
& && u(t) = u_k, \quad &&t_k \leq t < t_{k+1}, \, k \in \mathcal{N}, \\
& && \dot x(t) = A_c x(t) + B_c u(t), \quad && t_0 \leq t < t_0+T, \\
& && z(t) = C_c x(t) + D_c u(t), \, && t_0 \leq t < t_0+T, \\
& && \bar z(t) = \bar z_k,  &&t_k \leq t < t_{k+1}, \, k \in \mathcal{N}, \\
& && \tilde z(t) = z(t) - \bar z(t), && t_0 \leq t < t_0+T,
\end{alignat}
\end{subequations}
with the stage cost function
\begin{equation}
\begin{split}
    l_c(\tilde z(t)) &= \frac{1}{2} e^{-\mu t} \norm{ W_z \tilde z(t)  }_2^2
    = \frac{1}{2} e^{-\mu t} \left( \tilde z(t) ' Q_{c} \tilde z(t) \right), 
\end{split}
\label{eq:det-stagecost}
\end{equation}
where $x\in \mathbb{R}^{n_x \times 1}$ and $u \in \mathbb{R}^{n_u \times 1}$ are the state and input. $z \in \mathbb{R}^{n_z \times 1}$ and $\bar z \in \mathbb{R}^{n_z \times 1}$ are the output and its reference. The discount factor $\mu \in \mathbb{R}^+$ is a positive constant and the weight matrix $Q_{c} = W_z' W_z$ is semi-positive definite. 

Note that the above problem is in continuous-time with decision variables $x(t),u(t),z(t)$, and $\Tilde{z}(t)$. The control horizon $T = N T_s$ with sampling time $T_s$ and $N \in \mathbb{Z}^+$, and $\mathcal{N}=0,1,\ldots, N-1$. We assume piecewise constant input and output target, $u(t) = u_k$ and $\bar z(t) = \bar{z}_k$ for $t_k \leq t < t_{k+1}$. The corresponding discrete-time equivalent of~\eqref{eq:ContinuousTime-DiscountedLQOCP} is 
\begin{subequations}
\label{eq:DiscreteTime-DiscountedLQOCP}
\begin{alignat}{5}
& \min_{ x,u} \quad &&\phi = \sum_{k\in \mathcal{N}}
l_k(x_k,u_k)  \\
& s.t. && x_0 = \hat x_0, \\
& && x_{k+1} = A x_k + B u_k, \quad && k \in \mathcal{N},
\end{alignat}    
\end{subequations} 
with the stage costs
\begin{equation}
    l_k(x_k,u_k) = \frac{1}{2} \begin{bmatrix} x_k \\ u_k \end{bmatrix}' Q_k \begin{bmatrix} x_k \\ u_k \end{bmatrix} + (\overbrace{M_k \bar z_k}^{q_k})' \begin{bmatrix} x_k \\ u_k \end{bmatrix} + \rho_k, \, k \in \mathcal{N},
\label{eq:stageCost-noTimeDelay}
\end{equation}
where the quadratic and linear term matrices are $Q_k = e^{-\mu t_k}Q$ and $M_k = e^{- \mu t_k}M$. The constant term $\rho_k$ is 
\begin{equation}
    \rho_k = \frac{1}{2} \int_{t_k}^{t_{k+1}} e^{-\mu t} \bar z_k' Q_c \bar z_k dt = \frac{ e^{-\mu t_k} (1-e^{-\mu T_s})}{2 \mu} \bar z_k' Q_c \bar z_k.
    % & \Gamma_q(t) = e^{- \frac{\mu}{2}} e^{\begin{bmatrix}
    %     A_c & B_c \\ 0 & 0
    % \end{bmatrix}t},
    % \quad  
    % Q = \int_{t_k}^{t_{k+1}} \Gamma_q(t)' Q_c \Gamma_q(t) dt, 
    % \\
    % & 
    % \: M = - \int_{t_k}^{t_{k+1}} \Gamma_m(t)' Q_c dt ,
\end{equation}    

% Discounted LQ-OCP without time delay 
\begin{proposition}[Discretization of discounted LQ-OCPs] The system of differential equations
\label{prop:DiscretizationoftheLQOCP}
\begin{subequations}
\label{eq:LQDiscretization-MatrixODEsystem}
\begin{alignat}{3}
    \dot A(t) &= A_c A(t), \qquad && A(0) = I, \\
    \dot B(t) &= A(t) B_c, \qquad && B(0) = 0, \label{eq:B_nonDelay} \\
    \dot H_q(t) & = H_{cq} H_q(t), && H_q(0) = I_{xu}, \\
    \dot H_m(t) & = H_{cm} H_m(t), && H_m(0) = I_{xu}, \\
    \dot Q(t) &= H_q(t)' \bar Q_{c} H_q(t), \qquad && Q(0) = 0,  \\
    \dot M(t) &= H_m(t)' \bar M_{c}, \qquad && M(0) = 0,
\end{alignat}
where $I_{xu}$ is an identity matrix with size $n_{xu}= n_x + n_u$ and 
\begin{align}
    & H_c = \begin{bmatrix}
        A_c & B_c \\ 0 & 0
    \end{bmatrix}, 
    \: \: H_{cq} = H_c - \frac{\mu}{2} I_{xu}, 
    \: \: H_{cm} = H_c - \mu I_{xu},
    \\
    & \qquad \bar M_c =  - \begin{bmatrix} C_c & D_c \end{bmatrix}' Q_c,
    \qquad \bar Q_c = - \bar M_c \begin{bmatrix} C_c & D_c \end{bmatrix}, 
    % \\
    % & \Gamma_q(t) = \begin{bmatrix} C_c & D_c \end{bmatrix} H_q(t),
    % \qquad \Gamma_m(t) = \begin{bmatrix} C_c & D_c \end{bmatrix} H_m(t),
\label{eq:gamma(t)}
\end{align}
\end{subequations}
may be used to compute the discrete-time system matrices ($A = A(T_s)$, $B=B(T_s)$, $Q=Q(T_s)$, $ M=M(T_s)$) of discounted LQ-OCPs without time delays. 
\end{proposition}
\begin{proof}
    We can compute discrete-time state space matrices by $A(t)=e^{A_ct}$, $B(t)=\int_{0}^{t} A(s) B_c ds$. The solution of the state space model thus can be defined as 
    \begin{equation}
        \Gamma(t) = \begin{bmatrix}
            A(t) & B(t) \\ 0 & I 
        \end{bmatrix} = e^{H_c t}, \; \;  z(t) = \begin{bmatrix}
            C_c & D_c 
        \end{bmatrix} \Gamma(t) \begin{bmatrix}
            x_k \\ u_k
        \end{bmatrix}.
    \label{eq:solutionOfStateSpace}
    \end{equation}
    Moreover, as $e^{-\mu t}$ is a scalar and $e^{-\mu t} = {e^{-\frac{\mu}{2}t}}' e^{-\frac{\mu}{2}t}$, we obtain the following expressing by replacing $z(t)$ with~\eqref{eq:solutionOfStateSpace} in $l_c(\tilde z(t))$
    \begin{subequations}
    \label{eq:HqHmQandM_nonDelay}
        \begin{align}
            & H_{q}(t) = \Gamma(t) e^{-\frac{\mu}{2}t} = e^{H_{cq}t}, \;\;\; H_{m}(t) = \Gamma(t) e^{-\mu t} = e^{H_{cm} t}, \label{eq:HqHm_nonDelay}
            \\
            & Q(t) = \int_{0}^{t} e^{-\mu s} {e^{H_cs}}' \bar{Q}_c e^{H_cs} ds = \int_0^{t} {e^{H_{cq}s}}' \bar Q_c e^{H_{cq}s} ds, 
            \\
            & M(t) = \int_{0}^{t} e^{-\mu s} {e^{H_cs}}' \bar{M}_c ds = \int_0^{t}  {e^{H_{cm}s}}' \bar M_c ds.
        \end{align}
        % with $H_{q} = \Gamma(t) e^{-\frac{\mu}{2}t} = e^{H_{cq}t}$ and $H_{m} = \Gamma(t) e^{-\mu t} = e^{H_{cm} t}$.
        % \begin{align}
        %     H_{cq}(t) = e^{H_c t} e^{-\frac{\mu}{2}t} = e^{H_{cq}t},
        %     H_{cm}(t) = e^{H_c t} e^{-\mu t} = e^{H_{cm} t}.
        % \end{align}
    \end{subequations}
\end{proof}
\begin{remark} \label{remark:MatrixDiscountFactor}
Note that when the discount factor $\mu$ becomes a diagonal matrix $\mathcal{M}=\text{diag}(\mu_1,\mu_2,\ldots,\mu_{n_z})$ for $l_c(\tilde z(t))=0.5\tilde z(t)' (e^{- \mathcal{M}t}Q_c) \tilde z(t)$, the corresponding system of differential equations~\eqref{eq:LQDiscretization-MatrixODEsystem} becomes
\begin{subequations}
\begin{alignat}{3}
    \dot A(t) &= A_c A(t), \qquad && A(0) = I, \\
    \dot B(t) &= A(t) B_c, \qquad && B(0) = 0, \\
    \dot H_q(t) & = H_q(t) H_{c} - 0.5\mathcal{M}H_q(t), \quad && H_q(0) = \begin{bmatrix}
        C_c & D_c
    \end{bmatrix}, \\
    \dot H_{m}(t) & = H_m(t) H_{c} - \mathcal{M}H_m(t), \quad && H_m(0) = \begin{bmatrix}
        C_c & D_c
    \end{bmatrix},
    \\
    \dot Q(t) &= H_q(t)' Q_c H_q(t), \qquad && Q(0) = 0,  \\
    \dot M(t) &= -H_m(t)' Q_c, \qquad && M(0) = 0,
\end{alignat}
\end{subequations}
where $H_q = e^{-0.5\mathcal{M}t} [ C_c \; \; D_c] e^{H_c t}$ and $H_m = e^{-0.5\mathcal{M}t} H_q$. 
% \begin{subequations}
%     \begin{align}
%         & H_q(t) = e^{-0.5\mathcal{M}t} \begin{bmatrix}
%             C_c & D_c 
%         \end{bmatrix} e^{H_c t},
%         \\
%         & H_m(t) = e^{-\mathcal{M}t} \begin{bmatrix}
%             C_c & D_c 
%         \end{bmatrix} e^{H_c t}.
%     \end{align}
% \end{subequations}
% The corresponding discrete-time weight matrices are $Q_k=e^{-\mathcal{M}t_k}Q$, $m_k=e^{-\mathcal{M}t_k}M$ and $\rho_k$ is 
% \begin{align}
%     \rho_k = \frac{ e^{-\mathcal{M} t_k} (1-e^{-\mathcal{M} T_s})}{2 \mathcal{M}} \bar z_k' Q_c \bar z_k.
% \end{align}
\end{remark}
The differential equations described in Remark~\ref{remark:MatrixDiscountFactor} allow us to solve the discrete system matrices using standard ODE methods such as Euler and Runge-Kutta methods. In this paper, we only consider the case $e^{-\mu t}$ as a scalar.

%% %%%%%%%%%%%%%%%%%%%%%%%%%%%
\subsection{Discounted linear-quadratic optimal control problem with time delays}
Consider the following discounted LQ-OCP subject to a time-delay system 
\begin{subequations}
\label{eq:ContinuousTime-DiscountedLQOCP-TimeDelay}
\begin{alignat}{5}
& \min_{x,u,z,\tilde z} \: \: &&\phi = \int_{t_0}^{t_0+T} l_c(\tilde z(t)) dt 
\\
& s.t. && x(t_0) = \hat x_0, 
\\
& && u(t) = u_k, \: && t_k \leq t < t_{k+1}, \, k \in \mathcal{N}, 
\\
& && \dot x(t) = A_c x(t) + B_c u(t-\tau),  && t_0 \leq t < t_0+T, \label{eq:timeDelayStateSpace01}
\\
& && z(t) = C_c x(t) + D_c u(t-\tau), && t_0 \leq t < t_0+T, 
\label{eq:timeDelayStateSpace02}
\\
& && \bar z(t) = \bar z_k,  &&t_k \leq t < t_{k+1}, \, k \in \mathcal{N}, 
\\
& && \tilde z(t) = z(t) - \bar z(t), && t_0 \leq t < t_0+T,
\end{alignat}
\end{subequations}
where $\tau \in \mathbb{R}^{+}_0$ is the time delay. The stage cost function $l_c(\tilde z(t))$ is identical to~\eqref{eq:det-stagecost}. Note that when considering the MIMO system with multiple time delays and the system dynamic equations~\eqref{eq:timeDelayStateSpace01} and~\eqref{eq:timeDelayStateSpace02} become
\begin{subequations}
\label{eq:Deterministic-ContinuousTimeDelayStateSpaceMIMO}
    \begin{align} 
        \dot x_{ij}(t) &= A_{c,ij} x_{ij}(t) + B_{c,ij} u_j(t - \tau_{ij}), 
        \\ 
        z_{ij}(t) &= C_{c,ij} x_{ij}(t) + D_{c,ij} u_j(t - \tau_{ij}). 
    \end{align}
% and 
\begin{align}
    & \quad u=[u_1;u_2;\ldots,u_{n_u}],  &&x=[x_{11};x_{21};\ldots;x_{n_zn_u}],
    \\
    & \quad z=[z_{1};z_{2};\ldots;z_{n_z}], && z_i = \sum_{j=1}^{n_u} z_{ij},
\end{align}
\end{subequations}
where $A_{c,ij}$, $B_{c,ij}$, $C_{c,ij}$, $D_{c,ij}$ and $\tau_{ij}$ for $i=1,2,\ldots,n_z$ and $j = 1,2,\ldots, n_u$ are parameters of the $[i, j]$ SISO system describing the dynamics from the $j^{th}$ input to the $i^{th}$ output. 

% and the system state $x=[x_{11};x_{21};\ldots;x_{n_zn_u}]$ and the output 
% \begin{equation}
%     x = \begin{bmatrix}
%         x_{11} \\ x_{21} \\ \vdots \\ x_{n_z n_u}
%     \end{bmatrix}, \quad
%     z = \begin{bmatrix}
%         z_{1} \\ z_{2} \\ \vdots \\ z_{n_z},
%     \end{bmatrix}, \quad 
%     z_i = \sum_{j=1}^{n_u} z_{ij}.
% \end{equation}
Based on~\cite{zhaz2023LQDiscretizationWithDelays}, the discrete-time equivalent of~\eqref{eq:Deterministic-ContinuousTimeDelayStateSpaceMIMO} is 
\begin{subequations}
\begin{align}
    \overbrace{ \left [
                    \begin{array}{c}
                        x_{k+1} \\ \hdashline[2pt/2pt]
                        u_{o,k+1}
                    \end{array}  \right]  }^{=\Tilde{x}_{k+1}} 
    &= \overbrace{ 
         \left[
            \begin{array}{c;{2pt/2pt}c}
                A & B_{o,1} \\ \hdashline[2pt/2pt]
                0 & I_A
            \end{array}
         \right] }^{=\Tilde{A}}
        \overbrace{ \left [
                    \begin{array}{c}
                        x_k \\ \hdashline[2pt/2pt]
                        u_{o,k}
                    \end{array}  \right]  }^{=\Tilde{x}_{k}}
        + 
        \overbrace{
             \left[
                \begin{array}{c}
                    B_{o,2} \\ \hdashline[2pt/2pt]
                    I_B
                \end{array}
             \right] }^{=\Tilde{B}} u_k
        , \\ 
    z_k &= \overbrace{
            \left[
                \begin{array}{c;{2pt/2pt}c}
                    C_c & D_{o,1}
                \end{array}
             \right]
            }^{=\Tilde{C}} \Tilde{x}_k + 
            \overbrace{D_{o,2}}^{=\Tilde{D}} u_k,
\end{align} 
\label{eq:Deterministic-DiscreteTimeDelayStateSpaceSISO}%
\end{subequations}
where $u_{o,k}=[u_{k-\bar m},u_{k-\bar m+1},\ldots,u_{k-1}]$ is the historical input vector. $\bar m = \max \set{m_{ij}} $ indicates the maximum integer time delay constant with $\frac{\tau_{ij}}{T_s}=m_{ij} - v_{ij}$, where $m_{ij}\in \mathbb{Z}_0^+$ and $ 0 \leq v_{ij} < 1$ are integer and fractional time delay constants.

The discrete-time system matrices of~\eqref{eq:Deterministic-DiscreteTimeDelayStateSpaceSISO} are
\begin{subequations}
    % \begin{align}
    %     A = e^{A_c T_s}, \: B_{1}=\int_0^{T_s} e^{A_c t} B_{1c} dt, \: B_{2}=\int_0^{T_s} e^{V A_c t} \bar{B}_{2c} dt,
    % \end{align}
    % %
    \begin{align}
        & \bar{C}_j =\text{diag}(C_{c,1j}, C_{c,2j}, \ldots,C_{c,n_zj}), \: \:
        \text{for} \: \: j=1,2,\ldots,n_u,
        \\
        & \bar{D}_{c,i} = \sum_{j=1}^{n_u} D_{c,ij} e_j E_{\bar m+1}^{m_{ij}}, \qquad \qquad \: \, \: \: \text{for} \: \: i=1,2,\ldots,n_z.
        \\
        & C_c = [\bar C_1,\bar C_2,\ldots,\bar C_{n_u}], 
        \qquad \quad \; \; \:
        D_o = [\bar D_{1}; \bar D_2; \cdots; \bar D_{n_z}],
        \\
        & B_{o,1} = B_o(:, 1:end-n_u),  
        \qquad  B_{o,2} = B_o(:, \bar{m} n_u:end),
        \\
        & D_{o,1} = D_o(:, 1:end-n_u), 
        \quad \; \; \: D_{o,2} = D_o(:, \bar{m} n_u:end),
        \\
        & I_A = 
        \begin{bmatrix}
            0 & I & \ldots & 0 \\
            \vdots & \vdots & \ddots & \vdots \\
            0 & 0 & \ldots & I \\
            0 & 0 & \ldots & 0 
        \end{bmatrix},
        \qquad \quad  I_B = 
        \begin{bmatrix}
            0 \\ \vdots \\ 0 \\  I 
        \end{bmatrix}, \label{eq:IAandIB}
    \end{align}
\end{subequations}
and the matrices $A$ and $B_o = B_{1}+B_{2}$ are
\begin{subequations}
    \begin{align}
        & A_c  = \text{diag}(A_{c,11},A_{c,21},\ldots,A_{c,n_zn_u}), \quad 
        A(t) = e^{A_c t},
        \\
        & V=\text{diag}(V_{11},V_{21},\ldots,V_{n_zn_u}), \qquad \quad \; \; \; V_{ij} = I v_{ij},
        \\
        & B_{1c}  = \left[B_{1c,11}; B_{1c,21}; \ldots; B_{1c,n_zn_u}\right],
        B_{1c,ij} = B_c e_j E_{\bar m+1}^{m_{ij}},
        \\
        & B_{2c}  = \left[B_{2c,11}; B_{2c,21}; \ldots; B_{2c,n_zn_u}\right],
        B_{2c,ij} = B_c e_j E_{\bar m+1}^{m_{ij}+1},
        \\
        & B_{1}=\int_0^{T_s} e^{A_c t} B_{1c} dt, \; \; \; \;B_{2}= V \int_0^{T_s} e^{V A_c t} (B_{2c}-B_{1c}) dt,
    \end{align}
\label{eq:AandB1andB2}
\end{subequations}
where $e_j=[0,\ldots,1,\ldots,0]$ and $E_{\bar m+1}^{p}=[0,\ldots,I,\ldots,0]$ for $p = 1,2,\ldots,\bar{m}+1$ are the unit vector and matrix for selecting the $j^{th}$ historical input from the augmented input vector $\tilde u_k=[u_{o,k}; u_k]$ such that $u_{j,k-(\bar m+1)+p}=e_j E^p_{\bar m+1} \tilde u_k$ .

The discrete equivalent of the discounted LQ-OCP with time delays~\eqref{eq:ContinuousTime-DiscountedLQOCP-TimeDelay} has the same expressions as the non-delay case introduced in~\eqref{eq:DiscreteTime-DiscountedLQOCP} and~\eqref{eq:stageCost-noTimeDelay}, except that the system parameters become $\tilde x$, $\tilde A$, $\tilde B$, $\tilde C$, $\tilde D$. 
% The corresponding discrete-time LQ-OCP is
% \begin{subequations}
% \label{eq:Deterministic:DiscreteTime:LinearQuadraticOCP}
% \begin{alignat}{5}
% & \min_{x,u} \quad &&\phi = \sum_{k\in \mathcal{N}} l_k(x_k,u_k)  \\
% & s.t. && x_0 = \hat x_0, \\
% & && x_{k+1} = A x_k + B u_k, \quad && k \in \mathcal{N},
% \end{alignat}    
% \end{subequations}
% where the states $x$ and system matrices $A$, $B$, $C$, $D$ are in the augmented form $\tilde x, \tilde A, \tilde B, \tilde C, \tilde D$ described in~\eqref{eq:Deterministic-DiscreteTimeDelayStateSpaceSISO}. The stage cost function $l_k(x_k, u_k)$ is 
% \begin{equation}
%     l_k(x_k,u_k) = \frac{1}{2} \begin{bmatrix} x_k \\ u_k \end{bmatrix}' Q \begin{bmatrix} x_k \\ u_k \end{bmatrix} + q_k' \begin{bmatrix} x_k \\ u_k \end{bmatrix} + \rho_k, \quad k \in \mathcal{N},
% \label{eq:deterministic-stageCost}
% \end{equation}

% Discounted LQ-OCP with time delay 
\begin{proposition}[Discretization of discounted LQ-OCPs with time delays] The system of differential equations
\label{prop:DiscretizationoftheLQOCPWithTimeDelays}
\begin{subequations}
\label{eq:TimeDelayLQDiscretization-MatrixODEsystem}
\begin{alignat}{3}
    \dot A(t) &= A_c A(t), \qquad && A(0) = I, \\
    \dot A_v(t) &= VA_c A_v(t), \qquad && A_v(0) = I, \\
    \dot B_1(t) &= A(t) B_{1c}, \qquad && B_1(0) = 0, \\
    \dot B_2(t) &= A_v(t) \bar{B}_{2c}, \qquad && B_2(0) = 0, \\
    \dot H_q(t) & = H_{cq} H_q(t), && H_q(0) = I_{h}, \\
    \dot H_m(t) & = H_{cm} H_m(t), && H_m(0) = I_{h}, \\
    \dot Q(t) &= \Gamma_q(t)' \bar Q_{c} \Gamma_q(t), \qquad && Q(0) = 0,  \\
    \dot M(t) &= \Gamma_m(t)' \bar M_{c}, \qquad && M(0) = 0,
\end{alignat}
\end{subequations}
where $I_{h} = \text{diag}(I_{xu}, I_{xu}, I_{xu})$ is an identity matrix and 
\begin{subequations}
\label{eq:gamma(t)_timedelay}
\begin{align}
    H_{cq} &  =  H_c - \frac{\mathcal{\mu}}{2} I_h, \,
    H_{cm} = H_c - \mathcal{\mu} I_h, \,
    \bar{B}_{2c} = V(B_{2c} - B_{1c}),
    \\
    &  E_1 = \left[I_{xu}, I_{xu}, -I_{xu}  \right],
    \qquad  E_2 =  \left[I_{xu};  I_{xu}; I_{xu}  \right],
    \\
    & \Gamma_m(t) = E_1 H_m(t) E_2, \qquad  \bar M_c = - \begin{bmatrix} C_c & D_o \end{bmatrix} Q_c,
    \\
    & \Gamma_q(t) = E_1 H_q(t) E_2, \qquad \; \; \, \bar Q_c = - \bar{M}_c \begin{bmatrix} C_c & D_o \end{bmatrix}, 
\end{align}
\end{subequations}
may be used to compute ($A = A(T_s)$, $B_o=B_1(T_s)+B_2(T_s)$, $Q=Q(T_s)$, $ M=M(T_s)$) of discounted LQ-OCPs with time delays. 
\end{proposition}
\begin{proof}
In~\cite{zhaz2023LQDiscretizationWithDelays}, the solution of the time-delay state space~\eqref{eq:Deterministic-ContinuousTimeDelayStateSpaceMIMO} are defined as 
\begin{equation}
    z(t) = \begin{bmatrix}
        C_c & D_o
    \end{bmatrix} \Gamma(t) \begin{bmatrix}
        x_k \\ \tilde u_k 
    \end{bmatrix}, \quad \Gamma(t) = \begin{bmatrix} A(t) & B_o(t) \\ 0 & I \end{bmatrix},
\label{eq:solutionOfStateSpaceTimeDelayed}
\end{equation}
where $A(t)=e^{A_c t}$ and $B_o(t)=B_1(t)+B_2(t)$ are described in~\eqref{eq:AandB1andB2}. In this case, we cannot directly express $\Gamma(t)$ as a matrix exponential like the non-delay case introduced in Proposition~\ref{prop:DiscretizationoftheLQOCP} since $B_o$ consists of $B_1(t)$ and $B_2(t)$. However, we can decompose it into the linear combination of $A$, $A_v$, $B_1$ and $B_2$ as
\begin{equation}
    \begin{split}
        \Gamma(t) & = \overbrace{\begin{bmatrix}
            A(t) & B_1(t) \\ 0 & I 
        \end{bmatrix}}^{H_1(t)} + \overbrace{\begin{bmatrix}
            A_v(t) & B_2(t) \\ 0 & I 
        \end{bmatrix}}^{H_2(t)} - \overbrace{\begin{bmatrix}
            A_v(t) & 0 \\ 0 & I 
        \end{bmatrix}}^{H_3(t)}
        \\
        & = e^{H_{1c}t} + e^{H_{2c}t} - e^{H_{3c}t}
        \\
        & = E_1 H(t) E_2,
    \end{split}
\end{equation}    
with $H_c=\text{diag}\left(H_{1c},H_{2c},H_{3c}\right)$ and 
\begin{equation}
    H_{1c}= \begin{bmatrix}
        A_c & B_{1c} \\ 0 & 0 
    \end{bmatrix},
    H_{2c} = \begin{bmatrix}
        VA_{c} & B_{2c} \\ 0 & 0 
    \end{bmatrix},
    H_{3c} = \begin{bmatrix}
        VA_{c} & 0 \\ 0 & 0 
    \end{bmatrix},
\end{equation}
where $H(t)=e^{H_ct}=\text{diag}\left(H_{1}(t),H_{2}(t),H_{3}(t)\right)$ and $H_k(t)=e^{H_{kc}t}$ for $k=1,2,3$. 

Consequently, as $e^{-\mu t}={e^{-\frac{\mu}{2}}}'e^{-\frac{\mu}{2}}$ is a scalar, we will get 
\begin{subequations}
\begin{align}
    & H_q(t) = e^{H_{cq} t}, \qquad \; \; \Gamma_q(t)= \Gamma(t) e^{-\frac{\mathcal{\mu}}{2}t} = E_1H_q(t)E_2,
    \\
    &  H_m(t) = e^{H_{cm} t}, \qquad  \Gamma_m (t)= \Gamma(t) e^{-\mathcal{\mu} t} = E_1H_m(t)E_2,
    \\
    & Q(t) = \int_{0}^{t} e^{-\mu s} \Gamma(s)' \bar{Q}_c \Gamma(s) ds = \int_0^{t} \Gamma_{q}(s)' \bar Q_c  \Gamma_{q}(s) ds, 
    \\
    & M(t) = \int_{0}^{t} e^{-\mu s} \Gamma(s)' \bar{M}_c ds = \int_0^{t}  \Gamma_m(s)' \bar M_c ds.
\end{align}    
\end{subequations}
\end{proof}
\begin{remark}
    Note that the differential equations introduced in~\eqref{eq:LQDiscretization-MatrixODEsystem} and~\eqref{eq:TimeDelayLQDiscretization-MatrixODEsystem} are identical when time delays are $\set{\tau_{ij}}=0$, and the time delay constants become $\bar m = 0$ and $V=0$. This corresponding to $B_{1c}=B_c$, $\bar{B}_{2c}=0$ and $H_{2c}=H_{3c}=0$. The expressions introduced in Proposition~\ref{prop:DiscretizationoftheLQOCPWithTimeDelays} become
    \begin{subequations}
    \begin{align}
        & A_v(t) = I,  &&  B_o(t) = \int_0^{t} A(s)B_c ds, 
        \\ 
        & H(t) = e^{H_ct}, &&\Gamma(t) = E_1 H(t) E_2 = e^{H_{1c}t},
        \\ 
        & H_q(t)= H(t) e^{-\frac{\mu}{2}t}, 
        &&
        \Gamma_q(t) = E_1 H_q(t) E_2= e^{(H_{1c}-\frac{\mu}{2})t},
        \\
        & H_m(t)= H(t) e^{-\mu t}, 
        &&
        \Gamma_m(t) = E_1 H_m(t) E_2= e^{(H_{1c}- \mu)t},
    \end{align}    
    where $B_o$, $\Gamma_q$ and $\Gamma_m$ have the same expressions as the non-delay discrete system matrices $B$, $H_q$ and $H_m$ introduced in~\eqref{eq:B_nonDelay} and~\eqref{eq:HqHm_nonDelay}.
    \end{subequations}
\end{remark}
%% %%%%%%%%%%%%%%%%%%%%%%%%%%%

\section{Numerical Methods of LQ Discretization}
\label{sec:NumericalMethods}
This section introduces numerical methods for solving proposed systems of differential equations. 
\subsection{Fixed-time-step ordinary differential equation method}
% As we describe the discrete system matrices ($A$, $B_o$, $Q$, $M$) as differential equations, it is natural to use numerical methods such as explicit Euler and Runge-Kutta (RK) methods to solve them. 
Consider an s-stage ODE method with the number of integration steps $N \in \mathbb{Z}^+$ and the time step $\delta t = \frac{T_s}{N}$. Note that $a_{i,j}$ and $b_i$ for $i=1,2,\ldots,s$ and $j=1,2,\ldots,s$ are the Butcher tableau's parameters of the ODE method. We have
% ($A$,$A_v$,$B_1$,$B_2$,$H_q$,$H_m$,$Q$,$M$) as 
\begin{subequations}
    \label{eq:ODEmethods-numericalExpressions}
    % \begin{align}
    %      A_{k+1} &=  A_k + h \sum_{i=1}^s b_i A_c A_{k,i}=\Lambda A_k,
    %      \\
    %      B_{1,k+1} &=  B_{1,k} + h \sum_{i=1}^s b_i A_{k,i} B_{1c} = B_{1,k} + \Theta_1 A_k \tilde B_{1c},
    %      \\
    %      A_{v,k+1} &= A_{v,k} + h \sum_{i=1}^s b_i VA_c A_{v,k,i}= \Lambda_v A_{v,k}, 
    %      \\
    %      B_{2,k+1} &=  B_{2,k} + h \sum_{i=1}^s b_i A_{v,k,i} \bar{B}_{2c} =B_{2,k} + \Theta_2 A_k \tilde B_{2c}, 
    %      \\
    %      H_{q,k+1} &=  H_{q,k} + h \sum_{i=1}^s b_i H_{cq} H_{q,k,i}= \Omega_q H_{q,k}, 
    %      \\
    %      H_{m,k+1} &=  H_{m,k} + h \sum_{i=1}^s b_i H_{cm} H_{m,k,i} = \Omega_m H_{m,k}, 
    %      \\
    %      M_{k+1} &= M_k + h \sum_{i=1}^s b_i \Gamma_{m,k,i}' \bar M_c = M_k + E_2' H_{m,k}' \tilde M_c,
    %      \\
    %      Q_{k+1} &= Q_k + h \sum_{i=1}^s b_i \Gamma_{q,k,i}' \bar Q_c \Gamma_{q,k,i}= Q_k + E_2'H_{q,k}' \tilde Q_c H_{q,k}E_2,
    % \end{align}
    \begin{align}
         & A_{k+1} = \Lambda A_k,  && k \in \mathcal{N}, 
         \\
         & B_{1,k+1} =  B_{1,k} + \Theta_1 A_k \tilde B_{1c},  && k \in \mathcal{N}, 
         \\
         & A_{v,k+1} = \Lambda_v A_{v,k},  && k \in \mathcal{N}, 
         \\
         & B_{2,k+1} =B_{2,k} + \Theta_2 A_{v,k} \tilde B_{2c},   && k \in \mathcal{N}, 
         \\
         & H_{q,k+1} = \Omega_q H_{q,k},  
         && k \in \mathcal{N}, 
         \\
         & H_{m,k+1} = \Omega_m H_{m,k},  && k \in \mathcal{N}, 
         \\
         & M_{k+1} =  M_k + E_2' H_{m,k}' \tilde M_c, && k \in \mathcal{N}, 
         \\
         & Q_{k+1} = Q_k + E_2'H_{q,k}' \tilde Q_c H_{q,k}E_2, && k \in \mathcal{N}, 
    \end{align}
where 
\begin{align}
    & \tilde B_{1c}= \delta t B_{1c}, &&\tilde M_c = \delta t \sum_{i=1}^s b_i \Omega_{m,i}' E_1' \bar M_c, 
    \\
    & \tilde B_{2c} =  \delta t  \bar B_{2c},
    && \tilde Q_c = \delta t \sum_{i=1}^s b_i \Omega_{q,i}' E_1' \bar Q_c E_1 \Omega_{q,i}.
\end{align}
\end{subequations}
% where $\tilde B_{1c}= h B_{1c}$, $\tilde B_{2c} = h \bar B_{2c}$, $\tilde M_c = h \sum_{i=1}^s b_i \Omega_{m,i}' E_1' \bar M_c$ and $\tilde Q_c = h \sum_{i=1}^s b_i \Omega_{q,i}' E_1' \bar Q_c E_1 \Omega_{q,i}$.

The coefficients $\Lambda$, $\Lambda_v$, $\Theta_1$, $\Theta_2$, $\Omega_m$ and $\Omega_q$ are functions of the Butcher tableau's parameters. They can be computed as 
% \begin{align}
%     & \tilde B_{1c}= h B_{1c}, && \tilde M_c = h \sum_{i=1}^s b_i \Omega_{m,i}' E_1' \bar M_c
%     \\
%     & \tilde B_{2c} = h \bar B_{2c},  && \tilde Q_c = h \sum_{i=1}^s b_i \Omega_{q,i}' E_1' \bar Q_c E_1 \Omega_{q,i},
%     \\
%     & \Gamma_{q,k,i} = E_1 H_{q,k,i} E_2, && \Gamma_{m,k,i} = E_1 H_{m,k,i} E_2, 
% \end{align}
% where $\bar B_{1c}= h B_{1c}$, $\tilde B_{2c} = h \bar B_{2c}$ are constant matrices. The matrices $A_{k,i}$, $B_{1,k,i}$, $A_{v,k,i}$, $B_{2,k,i}$, $H_{q,k,i}$ and $H_{m,k,i}$ are stage variables of ($A$,$A_v$,$B_1$,$B_2$,$H_q$,$H_m$) and $\Gamma_{m,k,i}=E_1H_{m,k,i}E_2$ and $\Gamma_{q,k,i}=E_1 H_{q,k,i}E_2$. The constant coefficients $\Lambda$, $\Theta_1$, $\Lambda_v$, $\Theta_2$, $\Omega_q$, $\Omega_m$ are functions of the Butcher tableau's parameters $a_{i,j}$ and $b_i$ for $i=1,2,\ldots,s$ and $j=1,2,\ldots,s$. We can compute the stage variables and the constant coefficients as
\begin{subequations}
    \label{eq:ODEmethods-StageVariableCoefficients}
    \begin{alignat}{3}
        & A_{k,i} = A_k +  \delta t \sum_{j=1}^{s} a_{i, j} \dot A_{k,j} = \Lambda_{i} A_k,
        \\
        & B_{1,k,i} = B_{1,k} +  \delta t \sum_{j=1}^{s} a_{i,j} \dot B_{1,k,j} = B_{1,k} + \Theta_{1,i} A_k \tilde B_{1c},
        \\
        & A_{v,k,i} = A_{v,k} +  \delta t \sum_{j=1}^{s} a_{i, j} \dot A_{v,k,j} = \Lambda_{v,i} A_{v,k},
        \\
        & B_{2,k,i} = B_{2,k} +  \delta t \sum_{j=1}^{s} a_{i,j} \dot B_{2,k,j} = B_{2,k} + \Theta_{2,i} A_{v,k} \tilde B_{2c},
        \\
        & H_{m,k,i} = H_{m,k} +  \delta t \sum_{j=1}^{s} a_{i,j} \dot H_{m,k,j} = \Omega_{m,i} H_{m,k}, 
        \\
        & H_{q,k,i} = H_{q,k} +  \delta t \sum_{j=1}^{s} a_{i,j} \dot H_{q,k,j} = \Omega_{q,i} H_{q,k}, 
    \end{alignat}
\end{subequations}
% %
\begin{algorithm}[bt]
\caption{Fixed-time-step ODE method}
\label{alg:ODEmethod-LQDiscretization}
\begin{flushleft}
    \textbf{Input:} $(A_c, B_c, C_c, D_c, \tau, Q_c, T_s, N)$ \\
    \textbf{Output:} $(A, B_o, Q, M)$ 
\end{flushleft}
\begin{algorithmic}
\State Compute the step size $ \delta t  = \frac{T_s}{N}$
\State Compute system matrices ($B_{1c}$,$B_{2c}$,$\bar M_c$,$\bar Q_c$,$H_{cq}$,$H_{cm}$) using~\eqref{eq:gamma(t)_timedelay}
\State Set initial states ($k=0$, $A_k = I$, $A_{v,k}=I$, $B_{1,k}=0$, $B_{2,k}=0$, $H_{m,k}=I_h$, $H_{q,k}=I_h$, $Q_k=0$, $M_k=0$)
\State Compute stage variable coefficients ($\Lambda_i$, $\Lambda_{v,i}$, $\Theta_{1,i}$, $\Theta_{2,i}$, $\Omega_{m,i}$ and $\Omega_{q,i}$) using~\eqref{eq:ODEmethods-StageVariableCoefficients}
\State Compute constant coefficients ($\Lambda$, $\Lambda_v$, $\Theta_1$, $\Theta_2$, $\Omega_m$ and $\Omega_q$) using~\eqref{eq:ODEmethods-VariableCoefficients}
\While{$k < N$}
    \State {Use \eqref{eq:ODEmethods-numericalExpressions} to update $(A_{k}, B_{1,k}, A_{v,k}, B_{2,k}, H_{m,k}, H_{q,k}, Q_{k}, M_{k})$} 
    \State {Set $k \leftarrow k + 1$}
\EndWhile
\State Set
 ($A=A_k$, $B_o=B_{1,k}+B_{2,k}$, $Q=Q_k$, $M=M_k$)
\end{algorithmic}
\end{algorithm}%
and 
\begin{subequations}
\label{eq:ODEmethods-VariableCoefficients}
    \begin{alignat}{3}
        & \Lambda = I +  \delta t  \sum_{i=1}^s b_i A_c \Lambda_i, \quad  && \Lambda_v = I +  \delta t  \sum_{i=1}^s b_i V A_c \Lambda_{v,i},
        \\
        & \Theta_1 =  \sum_{i=1}^s b_i\Lambda_{i}, && \Omega_m = I_h +  \delta t \sum_{i=1}^sb_i H_{cm} \Omega_{m,i},
        \\
        & \Theta_2 = \sum_{i=1}^s b_i\Lambda_{v,i}, && \Omega_q = I_h +  \delta t  \sum_{i=1}^s b_i H_{cq} \Omega_{q,i}, 
    \end{alignat}
    % \begin{equation}
    %     B_{1,k,i} = B_{1,k} + h\sum_{j=1}^{s} a_{i,j} \dot B_{1,k,j} = B_{1,k} + \Theta_{1,i} A_k \bar B_{1c},
    % \end{equation}
    % \begin{equation}
    %     B_{2,k,i} = B_{2,k} + h\sum_{j=1}^{s} a_{i,j} \dot B_{2,k,j} = B_{2,k} + \Theta_{2,i} A_{v,k} \bar B_{2c},
    % \end{equation}
    % \begin{equation}
    %     B_{o,k,i} = B_{1,k,i} + B_{2,k,i}, \quad 
    %     \Gamma_{k,i} = \begin{bmatrix}
    %     A_{k,i} & B_{o,k,i} \\ 0 & I 
    % \end{bmatrix},
    % \end{equation}
\end{subequations}
where $\Lambda_i$, $\Lambda_{v,i}$, $\Theta_{1,i}$, $\Theta_{2,i}$, $\Omega_{m,i}$ and $\Omega_{q,i}$ are coefficients of stage variables $A_{k,i}$, $A_{v,k,i}$, $B_{1,k,i}$, $B_{2,k,i}$, $H_{m,k,i}$ and $H_{q,k,i}$. 

Consequently, we can obtain $A(T_s)=A_N$, $B_o(T_s)=B_{1,N}+B_{2,N}$, $Q(T_s)=Q_N$, $M(T_s)=M_N$ with constant coefficients $\Lambda$, $\Lambda_v$, $\Theta_{1}$, $\Theta_2$, $\Omega_m$ and $\Omega_q$ when using 
fixed-time-step ODE methods. Algorithm \ref{alg:ODEmethod-LQDiscretization} presents the fixed-time-step ODE method for the discretization of the discounted LQ-OCPs with time delays.

%% Step doubling method %%%%%%%%%%%%%%%%%%%%%%%
\subsection{Step-doubling method}
\begin{table}[tb]
    \centering
    \caption{Numerical expressions of the step-doubling method}%
    \label{tab:Stepdoubling-numericalExpressions}%
    \begin{tabular}{ p{0.9cm} p{2.6cm}  p{3.45cm}  }
    \hline
      & {Numerical expression}  & {Step-doubling function}    
    \\ \hline
    {$\Tilde{A}(N)$} & {$\bar{\Lambda}^N$}  & {$\Tilde{A}(\frac{N}{2}) \Tilde{A}(\frac{N}{2})$} 
    \\
    {$\Tilde{B}_o(N)$} & {$\displaystyle \sum_{i=0}^{N-1} \bar{\Lambda}^i$}  & {$\Tilde{B}_o(\frac{N}{2}) \left(I + \Tilde{A}(\frac{N}{2}) \right)$} 
    \\
    {$\Tilde{H}_m(N)$} & {$\Omega_m^N$} & {$\Tilde{H}_m(\frac{N}{2}) \Tilde{H}_m(\frac{N}{2})$} 
    \\
    {$\Tilde{H}_q(N)$} & {$\Omega_q^N$} & {$\Tilde{H}_q(\frac{N}{2}) \Tilde{H}_q(\frac{N}{2})$} 
    \\
    {$\Tilde{M}(N)$} & {$\displaystyle \sum_{i=0}^{N-1} \left(\Omega_m^i \right)'$} & {$\Tilde{M}(\frac{N}{2}) \left(I_{h} + \Tilde{H}_m(\frac{N}{2})' \right)$}
    \\
    {$\Tilde{Q}(N)$} & {$\displaystyle\sum_{i=0}^{N-1} \left( \Omega_q^i  \right)' \tilde{Q}_c \left( \Omega_q^i \right)$} & {$\Tilde{Q}(\frac{N}{2}) + \Tilde{H}_q(\frac{N}{2})' \Tilde{Q}(\frac{N}{2}) \Tilde{H}_q(\frac{N}{2})$} 
    \\ \hline
    \end{tabular}%
\end{table}%
\begin{algorithm}[tb]
\caption{Step-doubling method}
\label{algo:Stepdoubling-LQDiscretization}
\begin{flushleft}
    \textbf{Input:} $(A_c, B_c, C_c, D_c, \tau, Q_c, T_s, j)$ \\
    \textbf{Output:} $(A, B_o, Q, M)$ 
\end{flushleft}
\begin{algorithmic}
\State Compute the number of the integration step $N = 2^j$
\State Compute the step size $ \delta t = \frac{T_s}{N}$
\State Compute system matrices ($B_{1c}$,$B_{2c}$,$\bar M_c$,$\bar Q_c$,$H_{cq}$,$H_{cm}$) using~\eqref{eq:gamma(t)_timedelay}
\State Compute stage variable coefficients ($\Lambda_i$, $\Lambda_{v,i}$, $\Theta_{1,i}$, $\Theta_{2,i}$, $\Omega_{m,i}$ and $\Omega_{q,i}$) using~\eqref{eq:ODEmethods-StageVariableCoefficients}
\State Compute constant coefficients ($\Lambda$, $\Lambda_v$, $\Theta_1$, $\Theta_2$, $\Omega_m$ and $\Omega_q$) using~\eqref{eq:ODEmethods-VariableCoefficients}
\State Set initial states ($i=1, \Tilde{A}(i) = \bar \Lambda$, $\Tilde{B}_o(i) = I$,  $\Tilde{H}_m(i) = \Omega_m$,  $\Tilde{H}_q(i) = \Omega_q$, $\Tilde{M}(i) = I_{xu}$, $\Tilde{Q}(i) = \tilde{Q}_c$)
\While{$i \leq j$} 
    \State Update ($\Tilde{M}(i)$, $\Tilde{Q}(i)$) using equations from Table \ref{tab:Stepdoubling-numericalExpressions}
    \State Update ($\Tilde{A}(i)$, $\Tilde{B}_o(i)$,$\tilde H_{m}(i)$, $\tilde H_q(i)$) using equations from Table \ref{tab:Stepdoubling-numericalExpressions}
    \State Set $i \leftarrow i + 1$
\EndWhile
\State Use \eqref{eq:Stepdoubling_numericalExpressions} to compute 
 $(A, B_o, Q, M)$
\end{algorithmic}
\end{algorithm}
Consider the fixed-time-step ODE method with the integration step $N=2^j$ for $j\in \mathbb{Z}^+$ and the step size $\delta t=\frac{T_s}{N}$, the matrices 
\begin{subequations}
    \begin{align}
        & \Tilde{A}(N) = \bar{\Lambda}^N,                       && \Tilde{A}(1) = \bar{\Lambda}, 
        \\
        & \Tilde{B}_o(N) = \displaystyle \sum_{i=0}^{N-1} \bar{\Lambda}^i,  && \Tilde{B}(1) = I,       
        \\
        & \Tilde{H}_{m}(N) = \Omega_m^N ,                 && \Tilde{H}_{m}(1) = \Omega_m,    
        \\ 
        & \Tilde{H}_{q}(N) = \Omega_q^N ,                 && \Tilde{H}_{q}(1) = \Omega_q,    
        \\ 
        & \Tilde{M}(N) = \displaystyle \sum_{i=0}^{N-1} \left(\Omega_m^i \right)' , && \Tilde{M}(1) = I_{xu}, 
        \\
        & \Tilde{Q}(N) = \displaystyle \sum_{i=0}^{N-1} \left(\Omega_q^i \right)' \tilde{Q}_c \left(\Omega_q^i \right), && \Tilde{Q}(1) = \tilde{Q}_c, 
    \end{align}
    \label{eq:Stepdoubling_finalexpressions}
\end{subequations}
can be used to compute $(A, B_o, M, Q)$
\begin{subequations}
    \begin{align}
        & A(T_s) = \Tilde{A}(N)(1:n_x,1:n_x), && B_o(T_s) = \Theta_o \Tilde{B}_o(N) \tilde{B}_{oc}, \\ 
        & M(T_s) = E_2' \tilde{M}(N) \tilde{M}_c , && Q(T_s) = E_2' \Tilde{Q}(N) E_2,  
    \end{align}
    \label{eq:Stepdoubling_numericalExpressions}%
\end{subequations}
where $\bar \Lambda=\text{diag}(\Lambda,\Lambda_v)$, $\Theta_o=[\Theta_1, \Theta_2]$, $\tilde B_{oc}=[\tilde B_{1c}; \tilde B_{2c}]$ are constant. The coefficients $\Lambda$, $\Theta_1$, $\Lambda_v$, $\Theta_2$, $\Omega_m$, $\Omega_q$, $\tilde M_c$ and $\tilde Q_c$ are introduced in~\eqref{eq:ODEmethods-numericalExpressions} and~\eqref{eq:ODEmethods-VariableCoefficients}.

In~\cite{NewStepDoubling2010Nigham,ExponentialIntegrators2011Higham,ScalingSquaringRevisited2005Nigham}, a scaling and squaring algorithm is introduced for solving the matrix exponential problem. For $A(t)=e^{At}$, they use the $\frac{n}{2}^{th}$ step's result $A(\frac{n}{2} \delta t)$ as the initial state to compute the double step's result $A(n\delta t)=A(\frac{n}{2}\delta t)A(\frac{n}{2}\delta t)$ and repeat it until $n\delta t=T_s$. We can use the same idea to compute matrices introduced in~\eqref{eq:Stepdoubling_finalexpressions}, and it leads to the step-doubling method. Define $f(n)$ for $f\in \left[ \tilde A, \tilde B_o, \tilde H_m, \tilde H_q, \tilde M, \tilde Q\right]$, the step-doubling expression for computing $f(n)$ can be written as 
\begin{subequations}
    \begin{equation}
    f(1) \rightarrow f(2) \rightarrow f(4) \rightarrow \ldots \rightarrow f(\frac{N}{4})\rightarrow f(\frac{N}{2}) \rightarrow f(N),
    \label{eq:Stepdoubling_exampleofA(t)}
    \end{equation}
    where
    \begin{equation}
        f(n) = f(\frac{n}{2}) f(\frac{n}{2}), \qquad \qquad n \in [2,4, \ldots, \frac{N}{2}, N].
    \end{equation}
    \label{eq:Stepdoubling_numericalExpressionOfA(t)}%
\end{subequations}%
Table \ref{tab:Stepdoubling-numericalExpressions} describes the step-doubling expressions for ($\Tilde{A}$, $\Tilde{B}_o$, $\Tilde{H}_m$, $\Tilde{H}_m$, $\Tilde{M}$, $\Tilde{Q}$). The step-doubling method takes only $j$ steps to get the same result as the fixed-time-step ODE method with $N=2^j$ integration steps. Algorithm \ref{algo:Stepdoubling-LQDiscretization} describes the step-doubling method for the discretization of the discounted LQ-OCPs.

%% matrix exponential method %%%%%%%%%%%%%% 
\subsection{Matrix exponential method}
Based on the formulas described in~\cite{Moler1978NineteenDubiousWays,Moler2003NineteenDubiousWays25YearsLater,vanLoan1978MatrixExponential}, we can obtain ($A$, $B_o$, $M$, $Q$) by solving the following matrix exponential problems 
\begin{subequations}
\label{eq:matrixExponentials}
    \begin{alignat}{5}
        \begin{bmatrix}
            \Phi_{1, 11} & \Phi_{1, 12}\\
            0 & \Phi_{1,22}
        \end{bmatrix} &= \text{exp} \left(\begin{bmatrix}
                            -H_{cq}' & E_1' \bar Q_c E_1 \\
                            0   & H_{cq}
                        \end{bmatrix}t \right), 
        \\
        \begin{bmatrix}
            I & \Phi_{2, 12}\\
            0 & \Phi_{2,22}
        \end{bmatrix} &= \text{exp} \left(\begin{bmatrix}
                            0 & I \\
                            0 & H_{cm}'
                        \end{bmatrix}t \right), 
        \\
        \Phi_{3} &= \text{exp} \left(H_c t \right), 
    \end{alignat}
and the elements are
    \begin{align}
        & \Phi_{1,12} = H_q(-t)' \int_{0}^{t} H_q(s)' E_1' \bar{Q}_c E_1 H_q(s) ds,
        \\
        & \Phi_{1,22} = H_q(t) = e^{H_{cq}t}, 
        \\
        & \Phi_{2,12} = \int_0^t H_{m}(s) ds,
        \\
        & \Phi_{3} = H(t) = \text{diag}(H_1(t),H_2(t),H_3(t)),
    \end{align}
\end{subequations}
where the matrices $H_c$, $H_{cq}$, $H_{cm}$, $\bar Q_c$, $H_q$, $H_m$, $H_1$, $H_2$ and $H_3$ are introduced in Proposition~\ref{prop:DiscretizationoftheLQOCPWithTimeDelays}. 

Set $t=T_s$, we can compute ($A$, $B_o$, $M$, $Q$) as 
\begin{subequations}
\begin{align}
    & A = \Gamma(1:n_x, 1:n_x),
    && B_o = \Gamma(1:n_x, n_x+1:end), 
    \\
    & M = E_2' \Phi_{2,12} \bar M_c,
    && Q = E_2' \Phi_{1,22}' \Phi_{1,12} E_2,
\end{align}    
where $\Gamma = E_1\Phi_3E_2$.
\end{subequations}

\section{Numerical Experiments}
\label{sec:NumericalExperiments}
In this section, we will test and compare the proposed three numerical methods.

Consider a MIMO input-output model $Z(s)=G(s)U(s)$ with the following transfer functions 
\begin{equation}
    G(s) = \begin{bmatrix}
        \frac{1}{(1.5s+1)(3s+1)}e^{-0.1s} & \frac{-2(2s+1)}{3.4s+1}e^{-1.6s} \\ 
        \frac{-0.5}{2.3s+1}e^{-2.0s} & \frac{2.4}{(1.7s+1)(0.9s+1)}e^{-0.9s} 
    \end{bmatrix}.
\end{equation}
We can convert the above transfer functions into the state space models introduced in ~\eqref{eq:Deterministic-ContinuousTimeDelayStateSpaceMIMO}. The state space matrices are in observable canonical form. 

Consider the discounted LQ-OCP described in~\eqref{eq:ContinuousTime-DiscountedLQOCP-TimeDelay}. The weight matrix $Q_c=\text{diag}(1.0,2.0)$, the discount factor $\mu=0.2$, the control horizon $T=20$ [s] and the sampling time $T_s=1$ [s]. We use the symbolic toolbox in Matlab to compute the discrete system matrices ($A$, $B_o$, $Q$, $M$) with their analytic expressions described in Proposition~\ref{prop:DiscretizationoftheLQOCPWithTimeDelays}.  The results are used as the true solution for comparing the results from proposed numerical methods. 

Fig. \ref{fig:ErrorPlotsandCPUTime_LQDiscretization} describes the error and CPU time of the fixed-time-step ODE and the step-doubling method. We notice that the step-doubling method (dot plots) has the same error as the fixed-time-step ODE method (line plots) with the same discretization scheme. All tested methods have the correct convergence order (indicated by dashed lines). The CPU time of the two methods is indicated in the bar plots. We observe that the CPU time of the fixed-time-step ODE method increases as the integration step or the stage number increases. The step-doubling method's CPU time is stable at around 1.8 [ms]. 
\begin{figure}[t]
    \centering
    \includegraphics[width=0.49\textwidth]{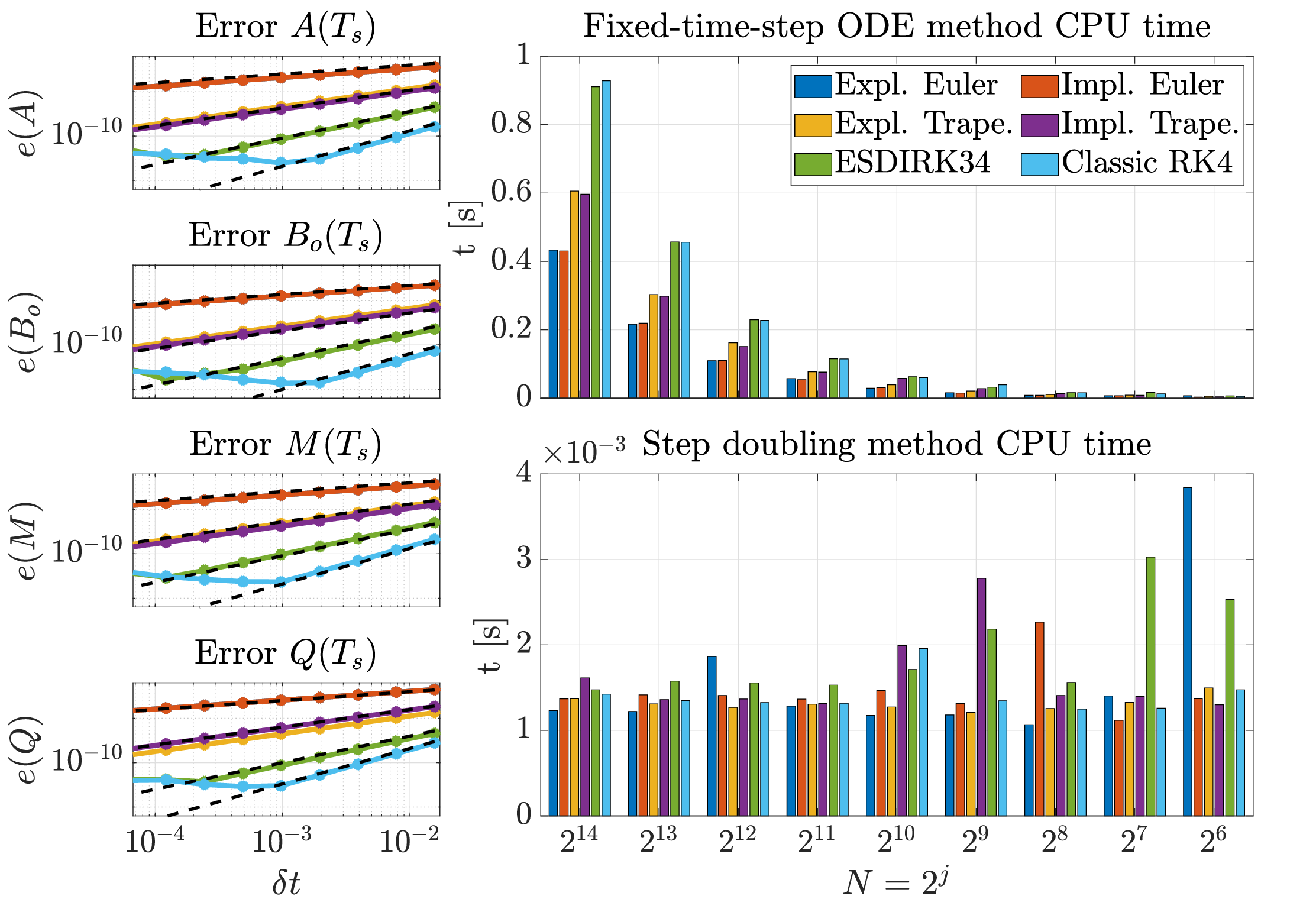}
    \caption{The error and CPU time of the fixed-time-step ODE method and the step-doubling method with different discretization schemes. The error is $e(i) = \norm{i(T_s) - i(N)}_{\infty}$ for $ i \in [A, B_o, M, Q]$, where 
    $i(T_s)$ is the result from the Matlab's symbolic toolbox.}
    \label{fig:ErrorPlotsandCPUTime_LQDiscretization}
\end{figure}
\begin{table}[b]
\centering
    \caption{CPU time and error of the scenario using classic RK4 with $N=2^{10}$}%
    \label{tab:NumericalExperiment-ErrorsandCPUTime}
    \begin{tabular}{ccccc}
    \hline
                & Unit & Matrix Exp. & \multicolumn{1}{l}{ODE Method} & \multicolumn{1}{l}{Step-doubling} \\ \hline
    $e(A)$      & [-] & $3.34 \cdot 10^{-16}$         & $9.12 \cdot 10^{-14}$         & $9.12 \cdot 10^{-14}$             
    \\
    $e(B_o)$    & [-] & $5.56 \cdot 10^{-17}$           & $8.33 \cdot 10^{-12}$   & $8.33 \cdot 10^{-12}$             
    \\
    $e(M)$      & [-] & $5.10 \cdot 10^{-16}$           & $6.88 \cdot 10^{-14}$   & $6.88 \cdot 10^{-14}$             
    \\
    $e(Q)$      & [-] & $8.10 \cdot 10^{-16}$            & $2.55 \cdot 10^{-13}$   & $2.55 \cdot 10^{-13}$             
    \\
    CPU Time    & [ms] & 44.1     & 60.3                         & 1.96      \\ \hline
    \end{tabular}
\end{table}

Table \ref{tab:NumericalExperiment-ErrorsandCPUTime} describes the error and CPU time of all three methods when using the classic RK4 with $j=10$ and $N=2^j$ (not for the matrix exponential). The matrix exponential method obtains the most precise result among all methods, while the other two methods have the same error. The fixed-time-step ODE method is the slowest, followed by the matrix exponential method. They spend 60.3 [ms] and 44.1 [ms], respectively. The step-doubling method is the fastest and only takes 1.96 [ms].

\section{Conclusions}
\label{sec:Conclusion}
This paper discussed the discretization of discounted LQ-OCPs with and without time delays. In the Propositions, the discrete system matrices of the discounted LQ-OCPs are described as systems of differential equations. Then, we introduced three numerical methods for solving the proposed systems of differential equations. All three numerical methods are tested and compared in the numerical experiment. Our results indicate that
\begin{itemize}
    \item [1.] All three methods can solve the proposed differential equations, and the matrix exponential method is the most precise. 
    \item [2.] The step-doubling method is significantly faster than
    the other two methods, keeping the same accuracy level as the fixed-time-step ODE method.
\end{itemize}

% Further, the cost of stochastic LQ-OCP adheres to a generalized $\chi^2$ distribution

% We have studied the discretization of LQ-OCP, and proposed three numerical methods for LQ discretization. Our methodologies ensure that the continuous-time LQ-OCP can be discretized by solving systems of differential equations $(A,B,R_{ww},Q,M)$, and the cost of stochastic LQ-OCP adheres to a generalized $\chi^2$ distribution. The numerical experiment tests and compares the proposed numerical methods, its result shows: 1) the step-doubling method is the fastest among all three numerical methods while retaining the same accuracy and convergence order as the ODE method, 2) the discrete-time LQ-OCP derived by the proposed numerical methods has the same distribution as the continuous-time LQ-OCP. It will be interesting to investigate the discretization of LQ-OCP with time delays, which is essential in practical applications of optimal control and estimation techniques.

% \newpage
% \input{tex/AppendixProofs}

%% %%%%%%%%%%%%%%%%%%%%%%%%%%%%%%%%%%%%%%%%%%%%%%%%%%%%%%%%%%%%%

\addtolength{\textheight}{-12cm}   % This command serves to balance the column lengths

%%%%%%%%%%%%%%%%%%%%%%%%%%%%%%%%%%%%%%%%%%%%%%%%%%%%%%%%%%%%%%%%%%%%%%%%%%%%%%%%

%%%%%%%%%%%%%%%%%%%%%%%%%%%%%%%%%%%%%%%%%%%%%%%%%%%%%%%%%%%%%%%%%%%%%%%%%%%%%%%%

%%%%%%%%%%%%%%%%%%%%%%%%%%%%%%%%%%%%%%%%%%%%%%%%%%%%%%%%%%%%%%%%%%%%%%%%%%%%%%%%

% \input{JanCopy}
% \input{tex/AppendixDistribution}

% \section*{ACKNOWLEDGMENT}
% XX

\bibliographystyle{IEEEtranS}
\bibliography{reference}

\begin{thebibliography}{10}
\providecommand{\url}[1]{#1}
\csname url@rmstyle\endcsname
\providecommand{\newblock}{\relax}
\providecommand{\bibinfo}[2]{#2}
\providecommand\BIBentrySTDinterwordspacing{\spaceskip=0pt\relax}
\providecommand\BIBentryALTinterwordstretchfactor{4}
\providecommand\BIBentryALTinterwordspacing{\spaceskip=\fontdimen2\font plus
\BIBentryALTinterwordstretchfactor\fontdimen3\font minus \fontdimen4\font\relax}
\providecommand\BIBforeignlanguage[2]{{%
\expandafter\ifx\csname l@#1\endcsname\relax
\typeout{** WARNING: IEEEtran.bst: No hyphenation pattern has been}%
\typeout{** loaded for the language `#1'. Using the pattern for}%
\typeout{** the default language instead.}%
\else
\language=\csname l@#1\endcsname
\fi
#2}}

\bibitem{NewStepDoubling2010Nigham}
A.~H. Al-Mohy and N.~J. Higham, ``A new scaling and squaring algorithm for the matrix exponential,'' \emph{SIAM Journal on Matrix Analysis and Applications}, vol.~31, no.~3, pp. 970--989, 2010.

\bibitem{ExponentialIntegrators2011Higham}
------, ``Computing the action of the matrix exponential, with an application to exponential integrators,'' \emph{SIAM Journal on Scientific Computing}, vol.~33, no.~2, pp. 488--511, 2011.

\bibitem{Cullum1969DiscreteApproximationstoContinuoustimeOCP}
J.~Cullum, ``Discrete approximations to continuous optimal control problems,'' \emph{SIAM Journal on Control}, vol.~7, no.~1, pp. 32--49, 1969.

\bibitem{Dontchev2001TheEulerApproximation}
A.~L. Dontchev and W.~W. Hager, ``The euler approximation in state constrained optimal control,'' \emph{Mathematics of Computation}, vol.~70, no. 233, pp. 173--203, 2001.

\bibitem{fu2018risk}
M.~Fu \emph{et~al.}, ``Risk-sensitive reinforcement learning via policy gradient search,'' \emph{arXiv preprint arXiv:1810.09126}, 2018.

\bibitem{Goebel2007LQRwithControlConstraint}
R.~Goebel and M.~Subbotin, ``Continuous time linear quadratic regulator with control constraints via convex duality,'' \emph{IEEE Transactions on Automatic Control}, vol.~52, no.~5, pp. 886--892, 2007.

\bibitem{granzotto2020finite}
M.~Granzotto, R.~Postoyan, L.~Bu{\c{s}}oniu, D.~Ne{\v{s}}i{\'c}, and J.~Daafouz, ``Finite-horizon discounted optimal control: stability and performance,'' \emph{IEEE Transactions on Automatic Control}, vol.~66, no.~2, pp. 550--565, 2020.

\bibitem{grune2015using}
L.~Gr{\"u}ne, W.~Semmler, and M.~Stieler, ``Using nonlinear model predictive control for dynamic decision problems in economics,'' \emph{Journal of Economic Dynamics and Control}, vol.~60, pp. 112--133, 2015.

\bibitem{Gu2003StabilityOT}
K.~Gu, V.~L. Kharitonov, and J.~Chen, \emph{Stability of Time-Delay Systems}.\hskip 1em plus 0.5em minus 0.4em\relax Birkhäuser, 2003.

\bibitem{Han2010ConvergenceofDiscretetime}
L.~Han, M.~Camlibel, J.~Pang, and W.~Heemels, ``Convergence of discrete-time approximations of constrained linear-quadratic optimal control problems,'' in \emph{49th IEEE Conference on Decision and Control (CDC)}, 2010, pp. 5210--5215.

\bibitem{Hansen1995DiscountedLinearExp}
L.~Hansen and T.~Sargent, ``Discounted linear exponential quadratic gaussian control,'' \emph{IEEE Transactions on Automatic Control}, vol.~40, no.~5, pp. 968--971, 1995.

\bibitem{Hendricks2008LCD}
E.~Hendricks, P.~H. Sørensen, and O.~Jannerup, \emph{Linear Systems Control: Deterministic and Stochastic Methods}.\hskip 1em plus 0.5em minus 0.4em\relax Berlin, German: Springer, 2008.

\bibitem{ScalingSquaringRevisited2005Nigham}
N.~J. Higham, ``The scaling and squaring method for the matrix exponential revisited,'' \emph{SIAM Journal on Matrix Analysis and Applications}, vol.~26, no.~4, pp. 1179--1193, 2005.

\bibitem{jacobson1973optimal}
D.~Jacobson, ``Optimal stochastic linear systems with exponential performance criteria and their relation to deterministic differential games,'' \emph{IEEE Transactions on Automatic control}, vol.~18, no.~2, pp. 124--131, 1973.

\bibitem{BAGTERPJORGENSEN2012187}
J.~B. Jørgensen, G.~Frison, N.~F. Gade-Nielsen, and B.~Damman, ``Numerical methods for solution of the extended linear quadratic control problem,'' \emph{IFAC Proceedings Volumes}, vol.~45, no.~17, pp. 187--193, 2012, 4th IFAC Conference on Nonlinear Model Predictive Control.

\bibitem{mena2022discounted}
H.~Mena, L.-M. Pfurtscheller, and M.~Voigt, ``Discounted cost linear quadratic gaussian control for descriptor systems,'' \emph{International Journal of Control}, vol.~95, no.~5, pp. 1349--1362, 2022.

\bibitem{Moler1978NineteenDubiousWays}
C.~Moler and C.~Van~Loan, ``Nineteen dubious ways to compute the exponential of a matrix,'' \emph{SIAM Rev.}, vol.~20, no.~4, p. 801–836, oct 1978.

\bibitem{Moler2003NineteenDubiousWays25YearsLater}
------, ``Nineteen dubious ways to compute the exponential of a matrix, twenty-five years later,'' \emph{SIAM Review}, vol.~45, no.~1, pp. 3--49, 2003.

\bibitem{vanLoan1978MatrixExponential}
C.~Van~Loan, ``Computing integrals involving the matrix exponential,'' \emph{IEEE Transactions on Automatic Control}, vol.~23, no.~3, pp. 395--404, 1978.

\bibitem{Alt2013ApproximationsofLQ}
F.~L. Walter~Alt, Robert~Baier and M.~Gerdts, ``Approximations of linear control problems with bang-bang solutions,'' \emph{Optimization}, vol.~62, no.~1, pp. 9--32, 2013.

\bibitem{wang2022system}
D.~Wang, J.~Ren, M.~Ha, and J.~Qiao, ``System stability of learning-based linear optimal control with general discounted value iteration,'' \emph{IEEE Transactions on Neural Networks and Learning Systems}, 2022.

\bibitem{zhaz2023LQDiscretizationWithDelays}
Z.~Zhang, S.~Hørsholt, and J.~B. Jørgensen, ``Numerical discretization methods for linear quadratic control problems with time delays,'' in \emph{12th IFAC Symposium on Advanced Control of Chemical Proccesses (ADCHEM 2024) (accepted)}, 2023.

\bibitem{zhaz2023LQDiscretization}
Z.~Zhang, J.~L. Svensen, M.~W. Kaysfeld, A.~H.~D. Andersen, S.~Hørsholt, and J.~B. Jørgensen, ``Numerical discretization methods for the extended linear quadratic control problem,'' in \emph{European Control Conference (ECC) 2024 (accepted)}, 2023.

\end{thebibliography}

\end{document}